\documentclass[preprint,12pt]{elsarticle}
\usepackage[T1]{fontenc}
\usepackage[utf8]{inputenc}
\usepackage{lmodern}
\usepackage{amsmath,amssymb,amsthm}
\usepackage{microtype}
\usepackage{xurl}
\usepackage{hyperref}
\usepackage{bookmark}
\numberwithin{equation}{section}

\newtheorem{theorem}{Theorem}[section]
\newtheorem{proposition}[theorem]{Proposition}
\newtheorem{lemma}[theorem]{Lemma}
\newtheorem{corollary}[theorem]{Corollary}
\newtheorem{example}[theorem]{Example}
\theoremstyle{definition}\newtheorem{definition}[theorem]{Definition}
\theoremstyle{remark}\newtheorem{remark}[theorem]{Remark}

\journal{Journal of Mathematical Analysis and Applications}
\hypersetup{pdftitle={Rigidity of Euclidean Minimal Hypersurfaces under Nonuniform Diagonal Dilations},pdfauthor={Jongha Lee, Suhwan Lee, Jae Won Lee},hidelinks}

\begin{document}
\begin{frontmatter}

\title{Rigidity of Euclidean Minimal Hypersurfaces under Nonuniform Diagonal Dilations}

\author[chosun]{Jongha Lee}
\ead{jhlee.eco@gmail.com}

\author[krei]{Suhwan Lee}
\ead{suhwan8352@krei.re.kr}

\author[gnu]{Jae Won Lee\corref{cor1}}
\ead{leejaew@gnu.ac.kr}
\cortext[cor1]{Corresponding author.}

\affiliation[chosun]{organization={Department of International Trade, Chosun University},
            addressline={76 Chosundae 5-gil, Dong-gu},
            city={Jeonnam-Gwangju},
            postcode={61452},
            country={Republic of Korea}}
\affiliation[krei]{organization={Korea Rural Economic Institute},
            addressline={601 Bitgaram-ro},
            city={Naju-si, Jeonnam-Gwangju},
            postcode={58217},
            country={Republic of Korea}}
\affiliation[gnu]{organization={Department of Mathematics Education and RINS, Gyeongsang National University},
            addressline={501 Jinju-daero},
            city={Jinju},
            postcode={52828},
            country={Republic of Korea}}

\begin{abstract}
Let $n\ge3$ and $D_t=\operatorname{diag}(t^{g_1},\ldots,t^{g_n})$ be a positive diagonal dilation family. We study connected embedded Euclidean hypersurfaces whose diagonal images are minimal. The level-set minimality operator splits into coefficients indexed by the pair sums $g_i+g_j$. Under pair-sum nonresonance, minimality at only $\binom n2$ distinct dilation parameters forces all pair coefficients to vanish. A dimension-reduction argument then shows, without any hypothesis on the coordinate components of the normal, that the second fundamental form vanishes identically. This yields an affine characterization. Repeated-weight helicoidal examples in every dimension and a resonant quadratic cone show that curvature cancellation can survive in genuinely nonuniform families. An application gives a finite-output-level rigidity criterion and an explicit representation for weighted-homogeneous production functions with minimal isoquants.
\end{abstract}

\begin{keyword}
minimal hypersurface \sep diagonal dilation \sep rigidity \sep resonance \sep second fundamental form \sep level-set mean curvature \sep quasi-homogeneous production function
\MSC[2020] 53A10 \sep 49Q05 \sep 35B06 \sep 91B38
\end{keyword}

\end{frontmatter}
\section{Introduction}\label{sec:intro}

Euclidean minimal hypersurfaces are invariant under rigid motions and uniform similarities; see, for example, \cite{Osserman1986,ColdingMinicozzi2011}. They are not generally preserved by nonconformal linear maps. This raises a simple rigidity question: what can a hypersurface look like if several members of a genuinely nonuniform diagonal dilation family are again minimal?

We consider
\begin{equation}\label{eq:dilation}
D_t=\operatorname{diag}(t^{g_1},\ldots,t^{g_n}),\qquad t>0,
\end{equation}
with positive weights $g_1,\ldots,g_n$. If all weights are equal, $D_t$ is an ordinary Euclidean dilation and no rigidity beyond minimality can be expected. The point of the present problem is therefore the interaction between Euclidean minimality and unequal coordinate scalings.

For a regular level hypersurface $\Sigma=\{F=0\}$, the useful quantity is
\begin{equation}\label{eq:minop-intro}
\mathcal M[F]=|\nabla F|^2\Delta F-\operatorname{Hess}F(\nabla F,\nabla F).
\end{equation}
Under \eqref{eq:dilation}, the first and second derivatives of the transported defining function acquire powers $t^{-g_i}$ and $t^{-(g_i+g_j)}$. After cancellation of the diagonal terms, $\mathcal M$ becomes a sum indexed by unordered pairs $\{i,j\}$, with exponent $2(g_i+g_j)$. This is the source of the arithmetic condition used below.

When all pair sums $g_i+g_j$ are distinct, the resulting exponential functions form a Chebyshev system. Consequently, $N=\binom n2$ distinct dilation parameters already separate all $N$ pair coefficients. A second step, which is independent of the arithmetic argument, shows that simultaneous vanishing of these coefficients forces the full second fundamental form to vanish. The latter implication is proved here without a nonvanishing assumption on the coordinate components of the normal; this removes a coordinate condition that is needed only by the single-chart recovery argument, not by the final rigidity mechanism.

The condition that the pair sums be distinct will be called \emph{pair-sum nonresonance}. The terminology is local to this paper. It is unrelated to anisotropic minimal-surface theory, where anisotropy usually refers to direction-dependent surface energies; compare \cite{MooneyYang2024}. Here the area functional is always Euclidean and the nonuniformity enters only through the ambient linear maps.

Several nearby questions have been studied from different viewpoints. Alc\'azar and Muntingh treat affine equivalence and symmetry detection for classes including minimal surfaces \cite{AlcazarMuntingh2022}, while Hasanis and Vlachos study hypersurfaces satisfying the spectral equation $\Delta x=Ax+B$ \cite{HasanisVlachos1992}. Neither setting asks for preservation of Euclidean minimality under finitely many members of a one-parameter nonuniform diagonal subgroup. The coefficient separation by the pair sums in the present argument appears to provide a different mechanism.

The last part of the paper applies this rigidity mechanism to weighted-homogeneous production functions. Differential-geometric production theory contains graph-based classifications for quasi-sum and homogeneous production models \cite{Chen2012QuasiSum,ChenVilcu2013}, while minimal isoquants have been studied for quasi-product, homogeneous, and quasi-sum models \cite{NeacsuEtAl2024,FuLuo2025}. Luo and Wang also classify minimal separable hypersurfaces arising in quasisum production models, including quadratic-cone and higher-dimensional catenoid types \cite{LuoWang2025}. Those results and the present one have different hypotheses: here the central requirement is preservation of Euclidean minimality along a weighted dilation orbit. Ordinary homogeneity, $g_1=\cdots=g_n$, is maximally resonant; the theorem below concerns the opposite regime in which the pair sums separate. This contrast explains why rich homogeneous minimal-isoquant families can coexist with the affine rigidity proved here. The resonant regime is not treated as a single opposite case: explicit examples below show failure of rigidity for repeated weights and for a four-index resonance, while a complete classification of resonant patterns is left open.

The paper has four sections and no subsidiary subdivision. Section 2 derives the level-set identities, proves finite coefficient separation, and establishes the dimension-reduction lemma that recovers the whole second fundamental form. Section 3 gives the affine characterization, identifies coordinate hyperplanes as the only invariant minimal examples in the nonresonant regime, develops resonant counterexamples, and applies the result to production functions. Section 4 records the precise scope of the result and the remaining resonance problem.

\section{Geometric framework and the diagonal-dilation identity}\label{sec:geometry}

The proof is most transparent at the level of a defining function. We first isolate the pairwise coefficients created by a diagonal dilation and verify that their vanishing is geometric. The arithmetic separation and the geometric recovery of the second fundamental form are then treated as two independent steps.

Throughout the paper, a hypersurface means an \emph{embedded} $C^2$ hypersurface unless stated otherwise. We assume $n\ge3$ in the main results; the base case $n=2$ will also be used in an induction argument.

\begin{definition}\label{def:regular}
Let $\Sigma\subset\mathbb R^n$ be a $C^2$ hypersurface. A $C^2$ function $F$ on a neighborhood $U$ of a point of $\Sigma$ is a \emph{regular defining function} if
\begin{equation}\label{eq:defining}
\Sigma\cap U=\{x\in U:F(x)=0\},\qquad \nabla F\ne0\quad\text{on }\Sigma\cap U.
\end{equation}
\end{definition}
Standard level-set terminology may be found in \cite{Lee2018}.

Set $\nu=\nabla F/|\nabla F|$. If $E_1,\ldots,E_{n-1}$ is a local orthonormal frame of $T\Sigma$, then, up to the choice of orientation,
\begin{equation}\label{eq:IIlevel}
II(E_\alpha,E_\beta)=\frac{\operatorname{Hess}F(E_\alpha,E_\beta)}{|\nabla F|},
\end{equation}
and hence
\begin{equation}\label{eq:meanlevel}
(n-1)H=\frac{|\nabla F|^2\Delta F-\operatorname{Hess}F(\nabla F,\nabla F)}{|\nabla F|^3}.
\end{equation}
We therefore set
\begin{equation}\label{eq:minop}
\mathcal M[F]=|\nabla F|^2\Delta F-\operatorname{Hess}F(\nabla F,\nabla F).
\end{equation}

\begin{definition}\label{def:minimal}
A $C^2$ hypersurface is \emph{Euclidean minimal} if its mean-curvature vector vanishes identically, equivalently if its scalar mean curvature is zero for either local choice of unit normal.
\end{definition}

\begin{proposition}\label{prop:minop}
A regular level hypersurface $\Sigma=\{F=0\}$ is minimal if and only if $\mathcal M[F]=0$ along $\Sigma$.
\end{proposition}
\begin{proof}
By \eqref{eq:meanlevel} and the regularity condition $|\nabla F|\ne0$, one has $H=0$ if and only if the numerator in \eqref{eq:meanlevel} vanishes. That numerator is precisely $\mathcal M[F]$.
\end{proof}

\begin{definition}\label{def:dilation}
For positive weights $g_1,\ldots,g_n$, the family \eqref{eq:dilation} is called a \emph{positive diagonal dilation family}. It is \emph{uniform} when all weights are equal and \emph{nonuniform} otherwise.
\end{definition}

The image $D_t\Sigma$ is locally defined on $D_tU$ by $F_t(y)=F(D_t^{-1}y)$. At $y=D_tx$,
\begin{align}
(F_t)_i(D_tx)&=t^{-g_i}F_i(x),\label{eq:firsttrans}\\
(F_t)_{ij}(D_tx)&=t^{-(g_i+g_j)}F_{ij}(x).\label{eq:secondtrans}
\end{align}
For $i<j$ define
\begin{equation}\label{eq:Cij}
C_{ij}[F]=F_j^2F_{ii}+F_i^2F_{jj}-2F_iF_jF_{ij}.
\end{equation}

\begin{proposition}[diagonal-dilation identity]\label{prop:dilationidentity}
For every $x\in\Sigma$,
\begin{equation}\label{eq:dilationidentity}
\mathcal M[F_t](D_tx)=\sum_{1\le i<j\le n}t^{-2(g_i+g_j)}C_{ij}[F](x).
\end{equation}
\end{proposition}
\begin{proof}
Equations \eqref{eq:firsttrans}--\eqref{eq:secondtrans} give
\begin{align}
|\nabla F_t|^2(D_tx)&=\sum_i t^{-2g_i}F_i^2,\label{eq:gradtrans}\\
\Delta F_t(D_tx)&=\sum_j t^{-2g_j}F_{jj}.\label{eq:laptrans}
\end{align}
Thus
\begin{equation}\label{eq:producttrans}
|\nabla F_t|^2\Delta F_t=\sum_{i,j}t^{-2(g_i+g_j)}F_i^2F_{jj},
\end{equation}
whereas
\begin{equation}\label{eq:hesstrans}
\operatorname{Hess}F_t(\nabla F_t,\nabla F_t)=\sum_{i,j}t^{-2(g_i+g_j)}F_iF_jF_{ij}.
\end{equation}
The $i=j$ terms cancel. Grouping $(i,j)$ and $(j,i)$ for $i\ne j$ gives \eqref{eq:Cij} and hence \eqref{eq:dilationidentity}.
\end{proof}

For later use set, for $i<j$,
\begin{equation}\label{eq:Vij}
V_{ij}=F_je_i-F_ie_j,
\end{equation}
and extend the notation by $V_{ji}=-V_{ij}$ when $j>i$. We also write $C_{ji}=C_{ij}$.

\begin{lemma}[pair-vector identities]\label{lem:pairvectors}
Along $\Sigma$, every $V_{ij}$ is tangent and
\begin{equation}\label{eq:CijII}
C_{ij}[F]=\operatorname{Hess}F(V_{ij},V_{ij})=|\nabla F|\,II(V_{ij},V_{ij}),
\end{equation}
up to orientation. Moreover,
\begin{equation}\label{eq:traceidentity}
\sum_{i<j}V_{ij}V_{ij}^{\!\top}=|\nabla F|^2I-\nabla F\nabla F^{\!\top},
\end{equation}
and consequently
\begin{equation}\label{eq:minop-pairs}
\mathcal M[F]=\sum_{i<j}\operatorname{Hess}F(V_{ij},V_{ij}).
\end{equation}
\end{lemma}
\begin{proof}
The tangency follows from $\langle V_{ij},\nabla F\rangle=F_jF_i-F_iF_j=0$. Expanding $\operatorname{Hess}F(V_{ij},V_{ij})$ gives \eqref{eq:Cij}, and \eqref{eq:IIlevel} gives \eqref{eq:CijII}. For \eqref{eq:traceidentity}, the diagonal $(k,k)$ entry of the left-hand side is $\sum_{j\ne k}F_j^2=|\nabla F|^2-F_k^2$, while the $(k,\ell)$ entry, $k\ne\ell$, is $-F_kF_\ell$. Taking the trace after multiplication by $\operatorname{Hess}F$ gives \eqref{eq:minop-pairs}.
\end{proof}

The vanishing of $C_{ij}$ does not depend on the chosen regular defining function. If $\widetilde F=\psi F$ with $\psi\ne0$ on $\Sigma$, then along $\Sigma$ one has $\widetilde F_i=\psi F_i$ and $\widetilde F_{ij}=\psi F_{ij}+\psi_iF_j+\psi_jF_i$. Direct substitution gives
\begin{equation}\label{eq:gauge}
\mathcal M[\widetilde F]=\psi^3\mathcal M[F],\qquad C_{ij}[\widetilde F]=\psi^3C_{ij}[F].
\end{equation}
Thus the equations $C_{ij}=0$ are geometric along the hypersurface.

\begin{definition}\label{def:nonresonance}
The weights $g_1,\ldots,g_n$ are \emph{pair-sum nonresonant} if
\begin{equation}\label{eq:nonresonance}
g_i+g_j=g_k+g_\ell\quad\Longrightarrow\quad\{i,j\}=\{k,\ell\}
\end{equation}
for unordered pairs $i<j$ and $k<\ell$. Equivalently, the $\binom n2$ pair sums are all distinct.
\end{definition}
For $n\ge3$, pair-sum nonresonance automatically implies that the weights themselves are pairwise distinct: if $g_a=g_b$ with $a\ne b$, then for any third index $c$ one has $g_a+g_c=g_b+g_c$. This is a sufficient arithmetic condition for coefficient separation; it is not asserted to be necessary for rigidity.

\begin{lemma}[finite coefficient separation]\label{lem:separation}
Assume pair-sum nonresonance and put $N=\binom n2$. If $D_{t_1}\Sigma,\ldots,D_{t_N}\Sigma$ are minimal for $N$ distinct positive numbers $t_1,\ldots,t_N$, then $C_{ij}[F]=0$ on $\Sigma$ for every $i<j$.
\end{lemma}
\begin{proof}
Fix $x\in\Sigma$ and list the distinct exponents $2(g_i+g_j)$ as $\lambda_1<\cdots<\lambda_N$. By Proposition~\ref{prop:dilationidentity},
\begin{equation}\label{eq:expsum}
P_x(t)=\sum_{r=1}^N a_r(x)t^{-\lambda_r}
\end{equation}
vanishes at $t_1,\ldots,t_N$. With $s=\log t$, write $Q_x(s)=\sum_{r=1}^Na_r(x)e^{-\lambda_rs}$. A nonzero exponential sum with $N$ distinct real exponents has at most $N-1$ distinct real zeros. This follows by induction: multiply by $e^{\lambda_1s}$, differentiate, and apply Rolle's theorem to reduce the number of terms by one. Since $Q_x$ has the $N$ distinct zeros $\log t_r$, it is identically zero. Hence every $a_r(x)$ vanishes, and nonresonance identifies these coefficients with the individual $C_{ij}[F](x)$.
\end{proof}

\begin{corollary}\label{cor:accumulation}
Under pair-sum nonresonance, the conclusion of Lemma~\ref{lem:separation} holds if $D_t\Sigma$ is minimal for every $t$ in any subset of $(0,\infty)$ having an accumulation point; in particular, it holds if $D_t\Sigma$ is minimal for every $t>0$.
\end{corollary}

We now remove the coordinate-normal hypothesis that would arise from trying to recover $II$ in a single fixed coordinate chart.

\begin{lemma}[coefficient vanishing forces flatness]\label{lem:allcoeffflat}
Let $n\ge2$ and let $\Sigma\subset\mathbb R^n$ be an embedded $C^2$ hypersurface. If, near each point of $\Sigma$, the conditions $C_{ij}[F]=0$ hold for one local regular defining function $F$ (equivalently, by \eqref{eq:gauge}, for every such defining function), for all $i<j$, then $II_\Sigma\equiv0$.
\end{lemma}
\begin{proof}
We argue by induction on $n$. For $n=2$, \eqref{eq:minop-pairs} reduces to $C_{12}=\mathcal M[F]$. Proposition~\ref{prop:minop} says that $\Sigma$ is a minimal curve in $\mathbb R^2$, hence locally a straight line and $II=0$.

Assume the result in dimension $n-1$ and let $n\ge3$. Put
\[
A=\{x\in\Sigma:II_x=0\}.
\]
This set is closed. Let
\[
\Omega=\{x\in\Sigma:\nu_i(x)\ne0\text{ for all }i\}.
\]
At a point of $\Omega$, choose an index $m$ with $F_m\ne0$ and define $W_i=F_me_i-F_ie_m$ for $i\ne m$. These vectors are tangent. They are linearly independent because $\sum_{i\ne m}a_iW_i=0$ forces $a_iF_m=0$ from the $e_i$-component for every $i\ne m$; hence they form a basis of $T_x\Sigma$. Equation \eqref{eq:CijII} gives $II(W_i,W_i)=0$. For $i,j\ne m$ we use the antisymmetric convention following \eqref{eq:Vij}, so the formula below is valid regardless of the ordering of $m,i,j$:
\[
F_mV_{ij}=F_jW_i-F_iW_j.
\]
Since $C_{ij}=0$, bilinearity and symmetry yield
\[
0=F_m^2II(V_{ij},V_{ij})=-2F_iF_jII(W_i,W_j).
\]
On $\Omega$ all factors are nonzero, so $II_x=0$. Hence $\Omega\subset A$.

It remains to treat the part not approximated by $\Omega$. Set $W=\Sigma\setminus\overline\Omega$. For each coordinate define
\[
Z_i=\{x\in W:\nu_i(x)=0\}.
\]
The closed sets $Z_i$ cover $W$. Since $W$ is an open subset of a finite-dimensional manifold, it is a Baire space; therefore the union of the relative interiors $\operatorname{int}_W Z_i$ is dense in $W$. Let $V$ be a connected open subset of some $\operatorname{int}_W Z_i$. On $V$, the constant vector $e_i$ is tangent everywhere. Restrict the constant ambient vector field $e_i$ to $V$. Its integral curves are the straight lines $x+s e_i$, and uniqueness for the induced local flow on the embedded hypersurface shows that sufficiently short pieces of these lines remain in $V$. After shrinking $V$ if necessary, $V$ is therefore a product $M\times I$ in the $e_i$-direction, where $M$ is a $C^2$ hypersurface in the coordinate hyperplane obtained by deleting the $i$-th coordinate and $I$ is an interval.

On such a product neighborhood one may choose a defining function $F(x)=G(\widehat x_i)$ independent of $x_i$. For pairs $j<k$ with $j,k\ne i$, the coefficients $C_{jk}[F]$ are exactly the coefficients $C_{jk}[G]$ of the base hypersurface $M$. They vanish by hypothesis. The induction assumption gives $II_M=0$. The cylindrical direction has zero second fundamental form and no mixed curvature, hence $II_V=0$. Thus $A$ contains $\Omega$ and a dense subset of $W$. Since $\Omega\cup W$ is dense in $\Sigma$ and $A$ is closed, $A=\Sigma$.
\end{proof}

\begin{remark}\label{rem:graph}
At a point where $F_n\ne0$, a local graph $x_n=h(x_1,\ldots,x_{n-1})$ gives $F=x_n-h$. The equations involving the $n$-th coordinate first control the pure second derivatives of $h$, while the remaining pair equations control the mixed terms wherever the corresponding first derivatives do not vanish. Lemma~\ref{lem:allcoeffflat} is the coordinate-free completion of this observation across the possible zero sets of those first derivatives.
\end{remark}

\section{Rigidity, resonance, and the production-function application}\label{sec:main}

The two steps from Section~\ref{sec:geometry} now fit together cleanly: nonresonance separates the pair coefficients, while Lemma~\ref{lem:allcoeffflat} turns their vanishing into flatness without any assumption on the normal coordinates. We then test the arithmetic hypothesis against explicit resonant examples and apply the result to weighted-homogeneous isoquants.

\begin{theorem}[finite-dilation rigidity and affine characterization]\label{thm:main}
Let $n\ge3$, let $\Sigma\subset\mathbb R^n$ be a connected embedded $C^2$ hypersurface, and let $D_t$ have positive pair-sum-nonresonant weights. Put $N=\binom n2$. The following are equivalent:
\begin{enumerate}
\item $D_t\Sigma$ is minimal for every $t>0$;
\item $D_{t_1}\Sigma,\ldots,D_{t_N}\Sigma$ are minimal for some $N$ distinct positive numbers $t_1,\ldots,t_N$;
\item $\Sigma$ is a connected open subset of an affine hyperplane.
\end{enumerate}
\end{theorem}
\begin{proof}
The implication (1)$\Rightarrow$(2) is immediate. Assume (2). If $F$ is a local regular defining function on $U$, then $F_t=F\circ D_t^{-1}$ is defined on $D_tU$ and
\[
\nabla F_t(D_tx)=D_t^{-1}\nabla F(x)\ne0.
\]
Thus Proposition~\ref{prop:dilationidentity} applies at each chosen parameter. Lemma~\ref{lem:separation} gives $C_{ij}=0$ for every pair, and Lemma~\ref{lem:allcoeffflat} yields $II\equiv0$.

The Weingarten equation implies that the normal line is locally constant. Connectedness then makes this normal line globally constant. Choosing a fixed unit vector $a$ spanning it, the function $x\mapsto a\cdot x$ has zero differential along $\Sigma$ and is therefore constant on the connected hypersurface. Hence $\Sigma$ is contained in an affine hyperplane $P$. Since $\Sigma$ is embedded, the inclusion $\Sigma\hookrightarrow P$ is a continuous injective map between $(n-1)$-manifolds. Invariance of domain shows that its image is relatively open in $P$. Thus (3) holds.

Conversely, a diagonal linear image of an affine hyperplane is again an affine hyperplane. Therefore every $D_t\Sigma$ is minimal, proving (3)$\Rightarrow$(1).
\end{proof}

\begin{example}[affine hyperplanes]\label{ex:affine}
If $P=\{x:a\cdot x=b\}$, then
\[
D_tP=\left\{y:\sum_i a_it^{-g_i}y_i=b\right\}.
\]
Thus every diagonal image of every affine hyperplane is again affine and minimal, including coordinate hyperplanes. This is the converse side of Theorem~\ref{thm:main}.
\end{example}

\begin{remark}\label{rem:global}
No completeness or properness assumption is used. The conclusion is an open connected piece of an affine hyperplane; equality with the whole hyperplane requires additional global hypotheses.
\end{remark}

\begin{remark}\label{rem:isotropic}
If $g_1=\cdots=g_n$, all pair sums coincide and nonresonance fails maximally. In that case $D_t$ is a Euclidean similarity and every minimal hypersurface remains minimal. Thus the nonresonant theorem is consistent with the familiar isotropic situation.
\end{remark}

\begin{proposition}\label{prop:3weights}
For $n=3$, pair-sum nonresonance is equivalent to pairwise distinct weights.
\end{proposition}
\begin{proof}
For example, $g_1+g_2=g_1+g_3$ if and only if $g_2=g_3$, and the other possibilities are identical after relabeling.
\end{proof}

\begin{corollary}[three dilations suffice in $\mathbb R^3$]\label{cor:3d}
Let $\Sigma\subset\mathbb R^3$ be a connected embedded $C^2$ surface and let $g_1,g_2,g_3$ be pairwise distinct positive weights. If $D_{t_1}\Sigma$, $D_{t_2}\Sigma$, and $D_{t_3}\Sigma$ are minimal for three distinct positive values $t_1,t_2,t_3$, then $\Sigma$ is an open subset of an affine plane.
\end{corollary}

For $n\ge4$, pairwise distinct weights do not imply pair-sum nonresonance; $(1,2,3,4)$ already satisfies $1+4=2+3$. The following examples show that such coincidences can correspond to genuine curvature cancellation.

\begin{proposition}[nonflat minimal geometry in resonant regimes]\label{prop:resonance}
Pair-sum resonance can preserve nonflat minimal geometry under genuinely nonuniform diagonal dilations.
\begin{enumerate}
\item In $\mathbb R^3$, let $g_1=g_2=1$, $g_3=\gamma>0$, $\gamma\ne1$, and on the quadrant $x_1,x_2>0$ set
\begin{equation}\label{eq:helicoid}
\Sigma_c=\{x_3=c\arctan(x_2/x_1)\},\qquad c\ne0.
\end{equation}
Then every $D_t\Sigma_c$ is again a helicoid patch and hence minimal, while $\Sigma_c$ is nonflat.
\item In $\mathbb R^4$, let $g=(1,2,3,4)$ and
\begin{equation}\label{eq:cone}
\Sigma=\{x_1x_4+x_2x_3=0\}.
\end{equation}
Its regular part is $\Sigma\setminus\{0\}$; it is connected, nonflat, and minimal, and it satisfies $D_t(\Sigma\setminus\{0\})=\Sigma\setminus\{0\}$ for every $t>0$.
\end{enumerate}
\end{proposition}
\begin{proof}
For (1), $D_t\Sigma_c$ is given by
\[
y_3=t^\gamma c\arctan(y_2/y_1),
\]
so it is again a helicoid patch. With $F=x_3-c\arctan(x_2/x_1)$ and $r^2=x_1^2+x_2^2$,
\[
C_{12}=0,\qquad C_{13}=-\frac{2cx_1x_2}{r^4},\qquad C_{23}=\frac{2cx_1x_2}{r^4}.
\]
Here the pair sums consist of the singleton value $g_1+g_2=2$ and the repeated value $g_1+g_3=g_2+g_3=1+\gamma$; these two values are distinct because $\gamma\ne1$. Thus the resonance classes are $\{\{1,2\}\}$ and $\{\{1,3\},\{2,3\}\}$, and the corresponding coefficient sums $C_{12}$ and $C_{13}+C_{23}$ both vanish. Equation \eqref{eq:dilationidentity} therefore gives minimality of every $D_t\Sigma_c$. Moreover, on the quadrant $x_1,x_2>0$ one has $C_{13}\ne0$; by \eqref{eq:CijII}, $II$ cannot vanish identically. Hence the helicoid patch is genuinely nonflat.

For (2), put $F=x_1x_4+x_2x_3$. Since $\nabla F=(x_4,x_3,x_2,x_1)$, the only singular point of the cone is the origin, so its regular part is $\Sigma\setminus\{0\}$. After an orthogonal change of coordinates the quadratic form has signature $(2,2)$; hence its link with the unit sphere is a Clifford torus $S^1\times S^1$, and $\Sigma\setminus\{0\}\cong(S^1\times S^1)\times(0,\infty)$ is connected. Since $F(D_tx)=t^5F(x)$, the regular part is $D_t$-invariant. Moreover $\Delta F=0$ and
\begin{equation}\label{eq:cone-hess}
\operatorname{Hess}F(\nabla F,\nabla F)=2x_1x_4+2x_2x_3=2F.
\end{equation}
Hence $\mathcal M[F]=0$ on $F=0$. Also
\[
C_{14}=-2x_1x_4,\qquad C_{23}=-2x_2x_3,
\]
while the other pair coefficients vanish, so $C_{14}+C_{23}=-2F=0$ on $\Sigma$. This realizes the resonance $g_1+g_4=g_2+g_3$. On the regular subset where $x_1x_4\ne0$, one has $C_{14}\ne0$, so \eqref{eq:CijII} again implies $II\ne0$ there; the cone is nonflat. Quadratic-cone minimal hypersurfaces also occur in the separable classification of Luo and Wang \cite{LuoWang2025}; the point here is specifically the compatibility of this cone with the resonant diagonal action.
\end{proof}

\begin{proposition}[repeated weights destroy rigidity in every dimension]\label{prop:repeatedweights}
Let $n\ge3$. If two dilation weights coincide, then there exists a connected real-analytic embedded nonflat minimal hypersurface patch whose image under every $D_t$ is again minimal. Thus, apart from pair-sum nonresonance, the example satisfies the geometric hypotheses appearing in Theorem~\ref{thm:main}; repeated weights are a structural resonant obstruction to its rigidity conclusion.
\end{proposition}
\begin{proof}
After relabeling, suppose $g_1=g_2$. For $n=3$ use the helicoid patch in Proposition~\ref{prop:resonance}(1). For $n\ge4$, choose constants $a_3,\ldots,a_{n-1}$ and $c\ne0$ and consider, on the quadrant $x_1,x_2>0$, the graph
\begin{equation}\label{eq:shearedhelicoid}
x_n=a_3x_3+\cdots+a_{n-1}x_{n-1}+c\arctan(x_2/x_1).
\end{equation}
Write $h=L+c\theta$, where $L=a_3x_3+\cdots+a_{n-1}x_{n-1}$ and $\theta=\arctan(x_2/x_1)$. Since $\theta$ is harmonic and $\operatorname{Hess}\theta(\nabla\theta,\nabla\theta)=0$, the graph minimal-surface operator satisfies
\[
(1+|\nabla h|^2)\Delta h-\operatorname{Hess}h(\nabla h,\nabla h)=0.
\]
Thus the graph is minimal. Under $D_t$, write $y=D_tx$. Solving for the last coordinate gives
\[
y_n=\sum_{k=3}^{n-1}a_k t^{g_n-g_k}y_k+c\,t^{g_n}\arctan\!\left(\frac{y_2}{y_1}\right),
\]
because
\[
\frac{x_2}{x_1}=t^{g_1-g_2}\frac{y_2}{y_1}=\frac{y_2}{y_1}
\]
uses precisely the hypothesis $g_1=g_2$. Hence each diagonal image is again a graph of the same helicoidal form, with rescaled linear coefficients and pitch, and is minimal. This is the point at which equality of the two weights is essential.

For an independent verification, apply the pair-vector identities of Lemma~\ref{lem:pairvectors} and the dilation identity. The coefficients satisfy
\[
C_{12}=0,
\]
\[
C_{k\ell}=0\quad(3\le k<\ell\le n-1),\qquad
C_{kn}=0\quad(3\le k\le n-1),
\]
and, for $3\le k\le n-1$,
\[
C_{1k}+C_{2k}=-a_k^2(h_{11}+h_{22})=0,
\]
while
\[
C_{1n}+C_{2n}=-(h_{11}+h_{22})=0.
\]
Because $g_1=g_2$, the pairs $\{1,k\}$ and $\{2,k\}$ always lie in the same resonance class. Any additional coincidences among the remaining weights merely merge classes whose displayed coefficient sums are already zero. Hence \eqref{eq:dilationidentity} independently confirms minimality of every diagonal image. Finally,
\[
C_{1n}=-h_{11}=-\frac{2cx_1x_2}{(x_1^2+x_2^2)^2}\ne0
\]
on the chosen quadrant, so \eqref{eq:CijII} proves that the second fundamental form is not identically zero.
\end{proof}

\begin{remark}\label{rem:resonance-open}
Every equality between two distinct unordered pair sums falls into one of two types. If the two pairs share an index, cancellation of that common weight forces two weights to coincide; Proposition~\ref{prop:repeatedweights} shows that this shared-index resonance destroys rigidity in every dimension. The only remaining possibility is a four-index resonance, in which the two pairs are disjoint. Proposition~\ref{prop:resonance}(2) shows that such a resonance can also destroy rigidity in dimension four. Thus, after the repeated-weight case is removed, the unresolved regime consists of systems with pairwise distinct weights that contain at least one four-index resonance between disjoint pairs. It remains to determine which grouped coefficient systems in this regime still force $II=0$ in higher dimensions or in the presence of several simultaneous four-index resonances.
\end{remark}

\begin{corollary}[$D_t$-invariant case]\label{cor:invariant}
Let $n\ge3$ and let the positive weights be pair-sum nonresonant. If a connected embedded $C^2$ hypersurface $\Sigma\subset\mathbb R^n$ is minimal and satisfies $D_t\Sigma=\Sigma$ for every $t>0$, then $\Sigma$ is an open connected subset of a coordinate hyperplane $\{x_i=0\}$ for some $i$.
\end{corollary}
\begin{proof}
Theorem~\ref{thm:main} first gives that $\Sigma$ is an open connected subset of an affine hyperplane $P=\{x:a\cdot x=b\}$. Since $D_t\Sigma=\Sigma$, the set $\Sigma$ is a nonempty relatively open subset of $P$ and is contained in $D_tP$. Two affine hyperplanes sharing such a relatively open set coincide, so $D_tP=P$. Equality of the two affine hyperplanes gives a scalar $\lambda(t)\ne0$ such that
\[
a_it^{-g_i}=\lambda(t)a_i\quad\text{for every }i,
\qquad b=\lambda(t)b.
\]
If $b\ne0$, then $\lambda(t)=1$ for all $t$, which would force $t^{-g_i}=1$ for every index with $a_i\ne0$, impossible because the weights are positive. Thus $b=0$. For any two indices with $a_i,a_j\ne0$ one then has $t^{-g_i}=t^{-g_j}$ for all $t$, so $g_i=g_j$. Pair-sum nonresonance implies pairwise distinct weights when $n\ge3$, hence exactly one coefficient of $a$ is nonzero. Therefore $P$ is a coordinate hyperplane.
\end{proof}

\begin{remark}
Conversely, every coordinate hyperplane is minimal and fixed by every $D_t$. More generally, a connected relatively open subset of a coordinate hyperplane satisfies the hypotheses of Corollary~\ref{cor:invariant} precisely when that subset is itself invariant under all $D_t$. Thus the corollary gives the exact geometric description of the $D_t$-invariant minimal hypersurfaces in the nonresonant regime; it does not assert that an arbitrary relatively open subset of a coordinate hyperplane is invariant.
\end{remark}

We now turn to the production application. Weighted homogeneity transports one isoquant to another, so the finite-dilation theorem becomes a finite-output-level criterion.

\begin{definition}\label{def:production}
Let $f:\mathbb R_{++}^n\to\mathbb R_{++}$ be $C^2$. We call $f$ a \emph{weighted-homogeneous production function} here if $f_i>0$ for all $i$ and
\begin{equation}\label{eq:weightedhom}
f(D_tx)=t^qf(x),\qquad q>0,
\end{equation}
for a positive diagonal dilation family. For $c>0$, the level set $M_c=f^{-1}(c)$ is called the \emph{isoquant} at output $c$; compare \cite{VilcuVilcu2019,NeacsuEtAl2024,Lloyd2012}. Its nonemptiness and regularity are consequences recorded in Lemma~\ref{lem:isoquanttransport}.
\end{definition}

\begin{lemma}[isoquant transport]\label{lem:isoquanttransport}
For every $c,t>0$,
\begin{equation}\label{eq:isoquanttransport}
D_t(M_c)=M_{t^qc}.
\end{equation}
Moreover, $M_c$ is nonempty and regular for every $c>0$, and its normal can be chosen with all coordinate components positive.
\end{lemma}
\begin{proof}
Choose any $x_0\in\mathbb R_{++}^n$. Since $t\mapsto t^qf(x_0)$ ranges over $(0,\infty)$, every level $c>0$ is attained along the orbit $D_tx_0$, so $M_c\ne\varnothing$. If $x\in M_c$, then $f(D_tx)=t^qc$, proving one inclusion in \eqref{eq:isoquanttransport}; applying $D_{t^{-1}}$ gives the reverse inclusion. Since every $f_i>0$, one has $\nabla f\ne0$, so every $M_c$ is regular, and $\nabla f/|\nabla f|$ has positive coordinate components.
\end{proof}

\begin{corollary}[finite-output-level isoquant rigidity]\label{cor:production}
Assume that the weights are pair-sum nonresonant and put $N=\binom n2$. Suppose there exist $N$ distinct output levels $c_1,\ldots,c_N>0$ for which the isoquants $M_{c_r}$ are Euclidean minimal. Then every isoquant is affine. More precisely, for any fixed reference level $c_0>0$ there are constants $A_i>0$ such that
\begin{equation}\label{eq:affineisoquant}
M_c=\left\{x\in\mathbb R_{++}^n:\sum_{i=1}^n A_i\left(\frac{c}{c_0}\right)^{-g_i/q}x_i=1\right\},
\end{equation}
and
\begin{equation}\label{eq:gaugerep}
f(x)=c_0\rho(x)^q,
\end{equation}
where $\rho(x)>0$ is the unique solution of
\begin{equation}\label{eq:gaugeeq}
\sum_{i=1}^nA_ix_i\rho(x)^{-g_i}=1.
\end{equation}
\end{corollary}
\begin{proof}
We first record the global topology of the isoquants, which is needed before applying Theorem~\ref{thm:main}. Let $S=\{z\in\mathbb R_{++}^n:z_n=1\}$. The map
\[
\Phi_c:S\to M_c,\qquad \Phi_c(z)=D_{(c/f(z))^{1/q}}z,
\]
is continuous. Its inverse is
\[
\Psi_c(x)=D_{x_n^{-1/g_n}}x,
\]
which is also continuous. Hence $M_c\cong S\cong\mathbb R_{++}^{n-1}$ and is connected for every $c>0$.

Choose one of the given levels as $c_*=c_1$ and set $t_r=(c_r/c_*)^{1/q}$. The $t_r$ are distinct, and Lemma~\ref{lem:isoquanttransport} gives $D_{t_r}M_{c_*}=M_{c_r}$. The connected hypersurface $M_{c_*}$ therefore satisfies the hypotheses of Theorem~\ref{thm:main}, so it is an open subset of an affine hyperplane. Since $f_i>0$, its constant normal line is represented by a vector with positive components; after normalization the hyperplane has the form $\sum_iB_ix_i=1$ with $B_i>0$.

Because $M_{c_*}$ is embedded in this hyperplane, invariance of domain makes it relatively open there. Since $f$ is continuous, it is relatively closed in the intersection of the hyperplane with $\mathbb R_{++}^n$. The latter set is convex and connected, so $M_{c_*}$ equals its full positive part.

Transport by \eqref{eq:isoquanttransport} shows that every $M_c$ is the full positive part of an affine hyperplane. In particular, transport first from $c_*$ to the chosen reference level $c_0$ and write that hyperplane as $\sum_iA_ix_i=1$. A further transport from $c_0$ to $c$ then gives \eqref{eq:affineisoquant}. Finally, for fixed $x$, the function $s\mapsto\sum_iA_ix_is^{-g_i}$ decreases strictly from $+\infty$ to $0$. Thus \eqref{eq:gaugeeq} has a unique positive solution $\rho(x)$. The point $D_{\rho(x)^{-1}}x$ lies on $M_{c_0}$, and weighted homogeneity gives \eqref{eq:gaugerep}.
\end{proof}

\begin{example}[a nonresonant three-input realization]\label{ex:production}
Take $g=(1,2,3)$, $q>0$, $c_0>0$, and $A_1,A_2,A_3>0$. Define $\rho(x)>0$ implicitly by
\begin{equation}\label{eq:threeinputrho}
A_1x_1\rho^{-1}+A_2x_2\rho^{-2}+A_3x_3\rho^{-3}=1
\end{equation}
and set $f(x)=c_0\rho(x)^q$. The pair sums are $3,4,5$, so the weights are pair-sum nonresonant. The left-hand side of \eqref{eq:threeinputrho} is strictly decreasing in $\rho$, which gives existence and uniqueness. Implicit differentiation yields
\[
\rho_{x_i}=\frac{A_i\rho^{1-g_i}}{\sum_{j=1}^3g_jA_jx_j\rho^{-g_j}}>0,
\]
so $f_i>0$. At output $c$, $\rho=(c/c_0)^{1/q}$ and the isoquant is the affine plane
\[
\sum_{i=1}^3A_i\left(\frac{c}{c_0}\right)^{-g_i/q}x_i=1.
\]
Thus the hypotheses and conclusion of Corollary~\ref{cor:production} are simultaneously realized by a concrete nonresonant model.
\end{example}

The economic interpretation is immediate from \eqref{eq:affineisoquant}. With the convention $\operatorname{MRTS}_{ij}=f_i/f_j$,
\begin{equation}\label{eq:mrts}
\operatorname{MRTS}_{ij}\big|_{M_c}=\frac{A_i}{A_j}\left(\frac{c}{c_0}\right)^{-(g_i-g_j)/q}.
\end{equation}
Hence pairwise marginal rates of technical substitution are constant along a fixed isoquant but can vary across output levels when the weights differ. This is an MRTS statement only; no identification with Allen--Uzawa or other substitution elasticities is intended. The upper contour sets are intersections of the positive orthant with half-spaces, so the represented technology is quasi-concave. The level-set and duality interpretation of isoquants is standard in production theory \cite{Shephard1970,Lloyd2012}.

\section{Discussion and concluding remarks}\label{sec:discussion}

The proof separates arithmetic from geometry. Pair-sum nonresonance makes the $\binom n2$ dilation exponents distinct, so the same number of minimal diagonal images determines every pair coefficient. Lemma~\ref{lem:allcoeffflat} then shows that no additional coordinate-normal assumption is needed: coefficientwise vanishing already forces $II=0$. In dimension three this means that three distinct diagonal images suffice whenever the three weights are distinct.

The resonant examples mark the limitation of coefficient separation. Proposition~\ref{prop:repeatedweights} shows that repeated weights support nonflat helicoidal examples in every dimension, while the four-index resonance $1+4=2+3$ supports the invariant quadratic cone \eqref{eq:cone} in dimension four. In the nonresonant regime, Corollary~\ref{cor:invariant} gives a sharp contrast: an invariant minimal hypersurface must be an open piece of a coordinate hyperplane. These examples establish genuine failure of rigidity in important resonant regimes. A resonance between two pairs that share an index is exactly a repeated-weight resonance and is therefore settled by Proposition~\ref{prop:repeatedweights}. The remaining unresolved case is four-index resonance between disjoint pairs. The dimension-four cone shows that such resonance can destroy rigidity, but a complete characterization of the higher-dimensional grouped coefficient systems that still force $II=0$ remains open.

For production functions, the finite-dilation theorem has a particularly concrete consequence: under pair-sum nonresonance, minimality at only $\binom n2$ distinct output levels forces the entire weighted-homogeneous isoquant family to be affine and yields the representation \eqref{eq:gaugerep}--\eqref{eq:gaugeeq}. Ordinary homogeneity lies at the maximally resonant end of the spectrum, which is consistent with the richer homogeneous minimal-isoquant families found by Fu and Luo \cite{FuLuo2025}. The production application is therefore a consequence of the geometric rigidity mechanism, not a replacement for existing production-function classifications.

\end{document}